\documentclass[sn-mathphys,Numbered]{sn-jnl}

\usepackage{graphicx}%
\usepackage{multirow}%
\usepackage{amsmath,amssymb,amsfonts}%
\usepackage{amsthm}%
\usepackage{mathrsfs}%
\usepackage[title]{appendix}%
\usepackage{xcolor}%
\usepackage{textcomp}%
\usepackage{manyfoot}%
\usepackage{booktabs}%
\usepackage{algorithm}%
\usepackage{algorithmicx}%
\usepackage{algpseudocode}%
\usepackage{listings}%

\theoremstyle{thmstyleone}%
\newtheorem{theorem}{Theorem}
\newtheorem{proposition}[theorem]{Proposition}%

\theoremstyle{thmstyletwo}%
\newtheorem{example}{Example}%
\newtheorem{remark}{Remark}%

\theoremstyle{thmstylethree}%
\newtheorem{definition}{Definition}%

\theoremstyle{thmstyletwo}
\newtheorem{lemma}{Lemma}

\theoremstyle{thmstyletwo}
\newtheorem{corollary}{Corollary}
\def\C{\mathbb{C}}
\begin{document}

\title[Normal and unitary operators]{On properties of normal operators and self-adjoint operators on smooth Banach spaces}


\author*[1]{\fnm{Mohammed} \sur{Shameem}}\email{mohammed\_p230121ma@nitc.ac.in}
\author[2]{\fnm{Deepesh} \sur{K. P.}}\email{deepeshkp@nitc.ac.in}



\affil*[1,2]{\orgdiv{Department of Mathematics}, \orgname{National Institute of Technology Calicut}, \orgaddress{\street{NIT Campus}, \city{Kozhikode}, \postcode{673601}, \state{Kerala}, \country{India}}}




\abstract{
This article introduces classes of normal and unitary operators on smooth Banach spaces, providing extensions of the classical notions of normal and unitary operators from Hilbert spaces to the smooth Banach space setting. The proposed class of normal operators contains, in particular, the class of self-adjoint operators on Banach spaces known in the literature.
In addition, we study several properties of self-adjoint operators on smooth Banach spaces, with emphasis on the norm, minimum modulus, numerical radius, and Crawford number, as well as the corresponding attainment properties and the relations among these quantities.
Further, we obtain characterisations and spectral properties of the newly introduced classes of normal and unitary operators. Our results demonstrate close analogies with the corresponding theory of self-adjoint, normal, and unitary operators on Hilbert spaces.}




\maketitle

\section{Introduction}\label{intro}
The study of operators on Hilbert spaces, especially their spectral properties, norm attainment, and minimum modulus attainment, has a long and rich history and remains a vibrant area of research (\cite{carvajal2012operators,carvajal2014operators,udayminimum2,chakraborty2020some,ganesh2015,venku2019absolutely}). In recent decades, considerable attention has also been given to the numerical range of operators, and related quantities such as the numerical radius and Crawford number~(\cite{bhuniabook,udayminimum2,choi2023crawford,crawford2024,crawfordattainset,crawfordattainsetc}). Among these operators, self-adjoint operators play a central role, with their spectral and numerical range properties having been extensively studied~(\cite{carvajal2012operators,carvajal2014operators,ganesh2015,limaye1996functional}).

Motivated by these developments, several authors have attempted to extend the notion of self-adjointness beyond the Hilbert spaces, particularly to the setting of Banach spaces (\cite{shtraus1978,wojcik2013}). In this direction, García-Pacheco (2020) introduced a notion of self-adjoint operators on smooth Banach spaces in \cite{banachSA}, using the duality mapping. The concept of numerical range has also been extended to Banach spaces using continuous linear functionals \cite{duncan1971}. Despite the availability of many analogues of Hilbert space results in this setting, several classical features—such as the convexity of the numerical range—need not hold in Banach spaces.

In the Hilbert space setting, characterisations were obtained for norm-attaining and minimum modulus attaining self-adjoint operators using eigenvalues~\cite{carvajal2012operators,carvajal2014operators,ganesh2015}. Here we consider the attainment of such quantities for operators on smooth Banach spaces. We also study the norm attainment of self-adjoint operators on smooth Banach spaces through the attainment of numerical radius, since these two quantities coincide for self-adjoint operators. We show that the Crawford number coincides with the minimum modulus for strongly normal operators, extending a known result on positive operators in the Hilbert space setting to smooth Banach spaces. Further, for a special class of operators, we characterise the attainment of the Crawford number in terms of eigenvalues, which also provides a characterisation of minimum modulus attainment in the setting of smooth Banach spaces. Our result in Theorem~\ref{mu_power} on the minimum modulus of powers of operators extends a known result in two directions: from Hilbert spaces to smooth Banach spaces, and from positive operators to the analogue of self-adjoint operators. Further, it can be viewed as a minimum modulus analogue of the known norm version proved in~\cite{banachSA}. Using this result, we obtain several results concerning the Crawford number of powers of operators in different cases. 

We introduce a new notion of normal operators in smooth Banach spaces. This class contains all self-adjoint operators on smooth Banach spaces as well as all normal operators on Hilbert spaces. We also introduce a notion of unitary operators in this framework, extending the classical concept. Several characterizations of these classes of operators are established.

We analyse several spectral properties of the self-adjoint and normal operators defined on smooth Banach spaces. In particular, in analogy with the Hilbert space case, we establish a relation between eigenvectors corresponding to distinct eigenvalues of a self-adjoint operator and show that the eigenspectrum is countable. We also present a few results concerning the spectrum of normal operators on smooth Banach spaces. 

The article is organized as follows. In the next section, we introduce the basic notation and preliminary results required for the development of the paper. Section 3 is devoted to the study of attainment and various properties of quantities associated with self-adjoint operators on smooth Banach spaces; we also present several results concerning powers of operators in this section. In Section 4, we introduce and investigate properties of normal and unitary operators in the setting of smooth Banach spaces. Section 5 focuses on the spectral properties of self-adjoint and normal operators in this setting. The paper concludes with a discussion of some open problems related to this work.

\section{Notations and Preliminaries}\label{prelim}
In this article, Banach spaces over the complex field $\mathbb{C}$ are denoted by the symbols $X, Y$, and the set of all bounded linear operators from $X$ to $Y$ is denoted by $B(X,Y)$. When $X=Y$, we use the abbreviation $B(X)$ for $B(X,X)$, and the notation $X^*$ stands for $B(X,\mathbb{C})$, the space of all bounded linear functionals. The unit sphere $\{ x \in X : \|x\| = 1 \}$ of a Banach space $X$ is denoted as $S_X$, and we use the notation $H$ for a complex Hilbert space. By an operator, we mean a bounded linear operator throughout this article, and we denote by $\mathbb{N}$ the set of all natural numbers.\vskip0.1cm

For an operator $T$, the spectrum $\sigma(T)$ is the set of all $\lambda \in \mathbb{C}$ for which $T - \lambda I$ does not have a bounded inverse. Some important subsets of $\sigma(T)$ include the eigenspectrum $\sigma_{eig}(T)$, consisting of all those $\lambda$ such that $T - \lambda I$ is not injective, and the approximate spectrum $\sigma_{app}(T)$, consisting of all those $\lambda$ such that $T - \lambda I$ is not bounded below. The spectral radius of $T$, denoted by $\rho(T)$, is defined as
$$
\rho(T) = \sup \{ |\lambda| : \lambda \in \sigma(T) \}.
$$ Two important quantities associated with an operator $T\in B(X)$ are the operator norm $\|T\|$ defined by
$$\|T\|=\sup\left\{\|Tx\|:\, x\in S_X\right\},$$
and the minimum modulus $\mu(T)$ defined by
$$
\mu(T) = \inf \{ \|Tx\| : x \in S_X \}.
$$
An operator $T$ is said to attain its minimum modulus (or the norm) if there exists an $x_0 \in S_X$ such that $\mu(T) = \|Tx_0\|$ (or $\|T\|=\|Tx_1\|, \mbox{ for some } x_1\in S_X$, in the case of the norm). For some studies on the attainment of the minimum modulus and the norm, one may refer \cite{carvajal2012operators,carvajal2014operators, crawford2024,venku2019absolutely}.
\vskip0.1cm
Another important set, closely connected to the spectrum of an operator, is the numerical range $W(T)$ of $T \in B(X)$, which is given by
$$
W(T) = \{ x^*(T(x)) : x \in S_X, \, x^* \in S_{X^*},\, x^*(x) = 1 \}.
$$
Some quantities associated with the numerical range of an operator $T$ are the numerical radius
$$
r(T) = \sup \{ |\lambda| : \lambda\in W(T) \}\,,
$$
and the Crawford number
$$
c(T) = \inf \{ |\lambda| : \lambda\in W(T) \},
$$
which form the upper and lower bounds (in  modulus) of the numerical range.
 An operator $T$ is said to attain its Crawford number (or the numerical radius) if there exists an $x_0 \in S_X$ such that $c(T) = | x_0^*(T(x_0))|$ (or $r(T) = |x_1^*(T(x_1))|$ for some $x_1\in S_X$, in the case of the numerical radius). One may refer to \cite{bhuniabook,udayminimum2,choi2023crawford, crawford2024,crawfordattainset,crawfordattainsetc} for studies on these numbers. An operator is called a normaloid operator if its numerical radius coincides with the operator norm.\vskip0.1cm

The transpose of an operator $T \in B(X,Y)$ is defined as the linear operator $T^{\prime}$ from $Y^*$ to $X^*$ defined by 
$$
(T^{\prime}(y))(x) = y(Tx), \quad y\in Y^*,\, x \in X,
$$
and for an operator $T \in B(H_1,H_2)$, where $H_1,H_2$ are Hilbert spaces, the adjoint is defined as the operator $T^*$ from $H_2$ to $H_1$ defined by
$$
\langle Tx, y \rangle = \langle x, T^*y \rangle,\, x \in H_1,\, y\in H_2.
$$

A Banach space $X$ is said to be smooth if every $x \in S_X$ admits a unique 
$x^* \in S_{X^*}$ such that $x^*(x)=1$ (\cite{dragomir2003}). Hence, for a smooth Banach space $X$, we can define a function $J : X \to X^*$, called the dual map or duality mapping given by $J(x)=x^*$. Thus the duality mapping $J$ satisfies 
\[\|J(x)\|=\|x\| \mbox{ and } J(x)(x) = \|x\|^2 \mbox{ for every } x\in X.\]
It is well known (\cite{dragomir2003}) that, in a smooth Banach space $X$, $J$ is surjective if and only if $X$ is reflexive and $J$ is injective if and only if $X$ is strictly convex. In a smooth Banach space $X$, for $x \in X$, we denote
\[
^\perp J(x) = \{y \in X : J(x)(y) = 0\}.
\]

In the setting of smooth Banach spaces, in \cite{banachSA}, the author introduced certain notions of self-adjoint, Hermitian, and related classes of operators, closely analogous to the corresponding concepts in the Hilbert space setting.
\begin{definition}(\cite{banachSA})
    Let $X$ be a smooth Banach space. An operator $T \in B(X)$ is said to be:
\begin{itemize}
    \item \textit{self-adjoint} if $T' \circ J = J \circ T$;
    \item \textit{Hermitian} if $J(x)(T(x)) \in \mathbb{R}$ for all $x \in X$;
    \item \textit{positive} if $J(x)(T(x)) \ge 0$ for all $x \in X$;
    \item \textit{strongly normal} if $T = S^2$ for some self-adjoint $S \in {B}(X)$.
\end{itemize}
\end{definition}
\noindent It is easy to observe that the concepts of self-adjoint and Hermitian operators (similarly, positive and strongly normal operators) coincide with the classical self-adjoint operator (positive operator, respectively) when $X$ is a Hilbert space over $\mathbb{C}$. Several properties of these operators were studied and many results on them were established in the article \cite{banachSA}.

\section{Attainment properties of self-adjoint operators}\label{self_adjoint}

In this section, we consider self-adjoint operators on smooth Banach spaces in the sense of \cite{banachSA}, and discuss about their norm, numerical radius, minimum modulus, Crawford numbers, their properties, inter relations and attainment. Note that, here onwards we use the notation $X$ to represent a smooth Banach space over $\mathbb{C}$.

Observe that for any operator $T$ on a Banach space, if $c(T)$ (or $r(T)$) is equal to the modulus of an eigenvalue $\lambda$ of $T$, then $T$ attains the Crawford number (numerical radius, respectively). This is because 
$$c(T)=|\lambda|=|\lambda| J(x_0)(x_0)=| J(x_0)(\lambda x_0)|=| J(x_0)(T x_0)|,$$
where $x_0$ is a unit eigenvector corresponding to the eigenvalue $\lambda$ of $T$.
 In the Hilbert space setting, it is known (Proposition 2.3, \cite{carvajal2012operators}) that a self-adjoint operator $T$ is norm attaining if and only if either $\|T\|$ or $-\|T\|$ is an eigenvalue of $T$. In the following proposition, this result is extended to the setting of smooth Banach spaces. It is worth observing that on a Hilbert space, for any self-adjoint operator $T$, the operator $\|T\|^{2}I - T^{2}$ is always strongly normal.

\begin{proposition}\label{SA_normattain}
Let $T$ be a self-adjoint operator on $X$ such that
$\|T\|^2 I - T^2$ is strongly normal. Then $T$ is norm attaining if and only if either $\|T\|$ or $-\|T\|$ is an eigenvalue of $T$.
\end{proposition}

\begin{proof}
Assume that $T$ is norm attaining. Then $
\|Ta\| = \|T\|$ for some $ a \in S_X.$ Hence, \(J(Ta)(Ta) = \|T\|^2 J(a)(a)\). Then
\[
\begin{aligned}
T' J(a)(Ta) = \|T\|^2 J(a)(a)\, \Rightarrow \, J(a)(T^2 a) = \|T\|^2 J(a)(a),
\end{aligned}
\]
which implies that \(J(a)(\|T\|^2 I - T^2)(a) = 0.\)

Let $
S^2 := \|T\|^2 I - T^2.$ Now,
\[J(a)(S^2 a)=S' J(a)(S a)=J(Sa)(Sa)=\|Sa\|^2.\]
Hence $S^2a=0$, implying $T^2a=\|T\|^2a$. But then
\[(T-\|T\|I)(T+\|T\|I)a = 0.\]
If $(T+\|T\|I)a = 0,$ then $-\|T\|$ is an eigenvalue of $T$; else  $\|T\|$ is an eigenvalue of $T$ with eigenvector $w=(T+\|T\|I)a$. The converse part is obvious.
\end{proof}

A similar characterization for minimum modulus attaining self-adjoint operators is given in the following proposition, whose proof is analogous to the proof of Proposition \ref{SA_normattain} (by taking $S^2 =T^2-\mu(T)^2I$). This generalizes the result given in Proposition 3.1 of \cite{ganesh2015} to smooth Banach spaces.
\begin{proposition}\label{minimum attain eigenvalue}
Let $T$ be a self-adjoint operator on $X$ such that
$ T^2-\mu(T)^2I$ is strongly normal. Then $T$ attains its minimum if and only if $\mu(T)$ or $-\mu(T)$ is an eigenvalue of $T$.
\end{proposition}
To prove a similar attainment result on the numerical radius, we first observe a result proved in \cite{banachSA} for self-adjoint operators.
\begin{theorem}(Corollary 3.2 , \cite{banachSA})\label{Pacheco-result}
    For a self-adjoint operator $T\in B(X)$, the numerical radius coincides with the spectral radius and the norm.
\end{theorem}
In the case of normaloid operators, the attainment of the norm follows from the attainment of the numerical radius as seen below.
\begin{proposition}\label{SA_norm_numer}
    Let $T\in B(X)$ be such that $r(T)=\|T\|$. Then $T$ attains its norm if it attains the numerical radius. In particular, for self adjoint operators, the attainment of numerical radius implies the attainment of norm.
\end{proposition}
\begin{proof}
     Suppose $T$ attains the numerical radius $r(T)$ at an $x_o \in S_X$. Then
     \[
     \begin{aligned}
         \|T\|=r(T)&=|J(x_o)(Tx_o)| \leq \|Tx_o\|\leq\|T\|.
     \end{aligned}\]
     Hence $T$ attains the norm at the same point $x_0\in S_X$. The particular case follows from Theorem \ref{Pacheco-result}.
\end{proof}
From the above proposition, it is clear that for normal operators on a Hilbert space, the numerical radius attainment implies the norm attainment. As a corollary to Proposition \ref{SA_norm_numer}, we obtain a characterization for the attainment of numerical radius/norm for the special type of self-adjoint operators described in Proposition \ref{SA_normattain} (which includes all self-adjoint operators in the Hilbert space setting) in terms of the eigenvalues.
\begin{corollary}\label{norm-attain-chara}
    Let $T$ be as in  Proposition \ref{SA_normattain}. Then the following statements are equivalent:
\begin{itemize}
    \item[(i)] $T$ attains the numerical radius
    \item[(ii)] $T$ attains the norm
    \item[(iii)] $\|T\|$ or $-\|T\|$ is an eigenvalue of $T$
\end{itemize}
\end{corollary}
\begin{proof}
    $(i)\Rightarrow (ii)$ follows from Proposition \ref{SA_norm_numer} and  $(ii)\Rightarrow (iii)$ follows from Proposition \ref{SA_normattain}. Note that $(iii)\Rightarrow (i)$ is true for any operator. 
\end{proof}

But the Crawford numbers can be different from the minium modulus even for self-adjoint operators on Hilbert spaces (as can be seen in the case of $F:\mathbb{K}^2\to \mathbb{K}^2$ given by $F(x_1,x_2)=(x_2,x_1)$), and the attainments also are not simultaneous (for example, $T(x_1, x_2,x_3,\ldots)=(x_2, x_1, \frac{x_3}{3}, \frac{x_4}{4}, \ldots)$ on $\ell^2(\mathbb{N})$ attains the Crawford number, but not the minimum modulus). However, for positive operators on a Hilbert spaces, close relations were established between the minimum modulus and the Crawford numbers (Proposition 2.2 \cite{carvajal2014operators}, Theorem 3.2 in \cite{crawford2024}). In the following propositions, we generalize these relations for certain classes of strongly normal operators on smooth Banach spaces, which are equivalent to positive operators - when considered on Hilbert spaces.

\begin{proposition}\label{min-equal-craw}
    If $T\in B(X)$ is such that $T-c(T)I$ is a strongly normal operator on $X$, then $\mu(T)=c(T)$.
\end{proposition}
\begin{proof}
For any $T \in B(X)$, we have $c(T) \leq \mu(T)$. Since $T-c(T)$I is strongly normal, $T$ is positive and  $c(T-c(T)I) =0$. So there exists a sequence $(x_n)$ in $S_X$ such that $J(x_n)((T-c(T)I)x_n)$ tends to $0$ as $n\to \infty$. Due to $$0\leq \|Tx_n\|-c(T) \leq \|(T-c(T)I)x_n\|,$$
if we show that $\|(T-c(T)I)x_n\|$ tends to $0$ as $n\to \infty$, we get $\mu(T)\leq c(T)$.
To show that, consider
\[J(x_n)((T-c(T)I)x_n) = J(x_n)(S^2x_n)=\|Sx_n\|^2.\]
Since $J(x_n)((T-c(T)I)x_n)$ tends to $0$, we get $Sx_n$ converges to $0$, and thus $\|(T-c(T)I)x_n\|$ tends to $0$, proving the result. 
\end{proof}

Note that, for any operator $T$, if $\mu(T)=0$, then since $c(T)\leq \mu(T)$, we have $c(T)=0$. Thus the minimum modulus coincides with the Crawford number for all such operators with  $\mu(T)=0$ (for example non-injective operators or compact operators on infinite-dimensional Banach spaces). For strongly normal operators on smooth Banach spaces, we have the following result in the other direction, as a corollary to the above proposition.

\begin{corollary}
    If $T\in B(X)$ is a strongly normal with $c(T)=0$, then $\mu(T)=0$.
\end{corollary}
In the following proposition, we provide a characterisation for the attainment of the Crawford number in terms of eigenvalues.

\begin{proposition}\label{attainment of crawford}
    Let  $T$ be as in Proposition \ref{min-equal-craw}. Then $T$ attains its Crawford number if and only if $c(T)$ is an eigenvalue of $T$.
\end{proposition}
\begin{proof}
    As we have mentioned, when $T-c(T)$I is a strongly normal operator, $T$ is positive. Suppose $T$ attains its Crawford number at some $x_o \in S_X$. Then taking $T-c(T)I=S^2$ for some self-adjoint operator $S$, we get
   \[0=J(x_o)((T-c(T)I)x_o) = J(x_o)(S^2x_o) = J(Sx_o)(Sx_o) = \|Sx_o\|^2.\] 
    This implies that $Tx_o = c(T)x_o$. The converse is true for any operator.
\end{proof}
The following corollary is now immediate.

\begin{corollary}
     Let $T$ be as in Proposition \ref{min-equal-craw}. If $T$ attains its Crawford number, then $T$ attains the minimum modulus.
\end{corollary}

Combining these results on the attainment of minimum modulus (Proposition \ref{minimum attain eigenvalue}) and Crawford number (Proposition \ref{attainment of crawford}), we have the following result analogous to Corollary \ref{norm-attain-chara}.
\begin{proposition}
    Let $T\in B(X)$ be such that both $T-c(T)I$ and $T^2-c(T)^2I$ are strongly normal. Then the following are equivalent:
    \begin{itemize}
        \item[i.] $T$ attains the minimum modulus
        \item[ii.] $T$ attains the Crawford number
        \item[iii.] $\mu(T)$ is an eigenvalue of $T$ 
    \end{itemize}
\end{proposition}

\begin{remark}
    If $T$ is a positive operator on a Hilbert space, then both $T-c(T)I$ and $T^2-c(T)^2I$ are strongly normal. Hence, the above proposition generalizes Proposition 3.4 of \cite{crawford2024}. 
\end{remark}

Regarding the norms of the powers of an operator, it is known that if $T$ is a normal operator on a Hilbert space, then $\|T^n\|\|T\|^n$ for each $n\in \mathbb{N}$ (\cite{nair2021functional}), and the numerical radius also has the property that $r(T^n)=r(T)^n$ (\cite{limaye1996functional}) for such an operator $T$. In \cite{banachSA}, it is proved that for a self-adjoint operator $T$ on a smooth Banach space, one has $\|T^{2^n}\|=\|T\|^{2^n}.$ Here, we establish the same result for more general powers of the self-adjoint operator using the spectral mapping theorem.

\begin{proposition}
    Let $T\in B(X)$ be a self-adjoint operator. Then $\|T^n\|=\|T\|^n$ for all $n \in \mathbb{N}$. In particular $r(T^n)=r(T)^n$.
\end{proposition}
\begin{proof}
    Since $T$ is self-adjoint, we have $\rho(T)=\|T\|$ (Theorem \ref{Pacheco-result}). Hence
    $$
    \begin{aligned}
        \|T\|^n = \rho(T)^n &= sup\{|\lambda| : \lambda \in \sigma(T) \}^n \\
        &= sup\{|\lambda| : \lambda \in \sigma(T^n) \} \\
        &= \rho(T^n) = \|T^n\|,
    \end{aligned}
    $$
    since $T^n$ is also self-adjoint. The particular case follows from $r(T)=\|T\|$. 
\end{proof}

Analogous to this, for the minimum modulus of a positive operator $T$ on a Hilbert space, it is proved (\cite{carvajal2014operators}) that $\mu(T^n)=\mu(T)^n$ for all $n\in \mathbb{N}$. We can extend this result to the set of all self-adjoint operators on reflexive smooth Banach spaces. Before proving the result, we establish a result on the invertibility of a self-adjoint operator.

\begin{lemma}\label{SA_non_inv_bound}
    A self-adjoint operator $T$ on a reflexive smooth Banach space $X$ is invertible if and only if it is bounded below.
\end{lemma}
\begin{proof}
    Suppose $T$ is not invertible. Then  $0 \in \sigma(T)=\sigma_{app}(T)\cup  \sigma_{eig}(T')$ (Proposition 13.8, \cite{limaye1996functional}).  If $0 \in \sigma_{eig}(T')$, then $\|T'f\|=0$ for some nonzero $f$ in $X^*$. Since $X$ is reflexive, $J$ is onto (Theorem 2, \cite{dragomir2003}), and hence
$$
\|T'f\| = \|T'J(x)\| = \|JTx\|=\|Tx\|,
$$
for some $x\neq 0$. This implies $\|Tx\|=0$, and so $0 \in \sigma_{eig}(T)\subseteq \sigma_{app}(T)$, proving that $T$ is not bounded below. The other part is clear. 
\end{proof}
Now we prove the result on minimum modulus of the powers of a self-adjoint operator, extending the existing result known on positive operators on Hilbert spaces, in this regard.
\begin{theorem}(compare Proposition 2.5, \cite{carvajal2014operators})\label{mu_power}
    For a self-adjoint operator $T$ on a reflexive smooth Banach space $X$, $\mu(T^{n})=\mu(T)^{n}$ for all $n \in \mathbb{N}$.
\end{theorem}
\begin{proof}
     We first consider the case when $T$ is invertible. It is known that $T^n$ and $T^{-1}$ is also self-adjoint (Proposition 2.2, \cite{banachSA}) and $\rho(T) = \|T\|$ (Theorem \ref{Pacheco-result}). Using Theorem 10.2.12 in \cite{nair2021functional}, we get
    $$
        \|T^{-1}\| = \frac{1}{dist(0,\sigma(T))}.
    $$
    Since $\mu(T) = \|T^{-1}\|^{-1}$, we have 
    $$
    \begin{aligned}
    \mu(T^n) = \|(T^n)^{-1}\|^{-1} &= dist(0,\sigma(T^n))\\
    &= inf\{|\lambda|:\lambda \in \sigma(T^n) \}\\
    &= dist(0,\sigma(T))^n\\
    &= \|T^{-1}\|^{-n}\\ &= \mu(T)^n.
    \end{aligned}
    $$
Now, if $T$ is not invertible, by Lemma \ref{SA_non_inv_bound}, we have $\mu(T)=0$. Since $T^n$ is self-adjoint and non-invertible, we get $\mu(T^n)=0$, proving the claim.
\end{proof}

Now, with respect to the Crawford numbers of powers of operators on Hilbert spaces, it is shown (\cite{crawford2024}) that $c(T^2)=c(T)^2$ for all positive operators, from which it follows that $c(T^{2^n})=c(T)^{2^n}$ for $n\in \mathbb{N}$. Here, we consider this problem for self-adjoint operators on reflexive smooth Banach spaces and obtain results that improve the existing results.
\begin{proposition}\label{even-power}
    For a self-adjoint operator $T$ on a reflexive smooth Banach space $X$, we have $c(T^{2n})=\mu(T^{2n}),\,n\in \mathbb{N}$. In addition, if $c(T)=\mu(T)$, then $c(T^{2n})=c(T)^{2n}$ for all $n \in \mathbb{N}$.
\end{proposition}
\begin{proof}
    Let $n\in \mathbb{N}$. Observe that 
    \[\|T^nx\|^2=J(T^nx)(T^nx)=T^{n\prime}J(x)(T^nx)=J(x)(T^{2n}x).\]
    Now, taking infimum over $x \in S_X$, we get $\mu(T^n)^2=c(T^{2n})$. This gives $c(T^{2n})=\mu(T^{2n})$ due to Theorem \ref{mu_power}. Now, if $c(T)=\mu(T)$, then we get
    \[c(T^{2n})=\mu(T^{2n})=\mu(T)^{2n}=c(T)^{2n},\]
    proving the claim.
\end{proof}
\begin{remark}
    If $X$ is a reflexive smooth Banach space, and $T$ is a strongly normal operator, say, $T=S^2$, on $X$ such that $\mu(S)=c(S)$, then from the proof of Proposition \ref{even-power}, it follows that 
    $c(T^n)=c(T)^n$ for all $n\in \mathbb{N}$.
\end{remark}
The result in Proposition \ref{min-equal-craw} leads to the following corollary.
\begin{corollary}
    Let $T$ be a self-adjoint operator on a reflexive smooth Banach space such that  $T-c(T)I$ is strongly normal. Then $c(T^{2n})=c(T)^{2n}$, $n \in \mathbb{N}$.
\end{corollary}

The example given after Proposition 3.6 in \cite{crawford2024} shows that $c(T^2)\neq c(T)^2$ is possible for a self-adjoint operator $T$ on a Hilbert space, if $c(T)\neq \mu(T)$. We conclude this session by providing an example to show that the result in Theorem \ref{mu_power} need not hold if the operator is not self-adjoint.
\begin{example}
    Consider $T : (\C^2,\|\cdot\|_2) \rightarrow (\C^2, \|\cdot\|_2)$ given by \[T(x_1,x_2)=(x_1+x_2,x_2),\quad (x_1,x_2)\in \C^2.\] 
    Since $T^*T$ is a positive operator on a finite-dimensional space, we have \[c(T^*T) = min\{\lambda : \lambda \in \sigma(T^*T)\}=\mu(T)^2,\] 
    by Proposition 3.1 in \cite{crawford2024}. Since the smallest eigenvalue of $T$ is $\frac{3-\sqrt{5}}{2}$, we get $c(T^*T) = \frac{3-\sqrt{5}}{2} = \mu(T)^2$. By a similar computation, we get $\mu(T^2)=c({T^2}^*T^2) = 3-2\sqrt{2}.$ Thus $\mu(T^2) \neq \mu(T)^2$.
\end{example}

\section{Classes of normal operators and unitary operators}
Due to the non-linearity of the duality map, extending the concept of normal operators to smooth Banach spaces in a conventional manner is difficult (\cite{banachSA}). Here we introduce the notion of a class of normal operators on smooth Banach spaces, which contains the class of all self-adjoint operators. We also show that this class of operators satisfy many of the known properties of the classical normal operators, defined on Hilbert spaces.
\begin{definition}\label{normal-defn}
    Let $X$ be a smooth Banach space and $T \in B(X)$. Then $T$ is said to be a normal operator if 
    $$
    \|Tx\| = \|T'Jx\| \, \mbox{ for all } x \in X.
    $$
\end{definition}

Analogous to the notion of unitary operators on Hilbert spaces, we also introduce the concept of unitary operators on smooth Banach spaces.
\begin{definition}\label{unitary-defn}
    Let $X$ be a smooth Banach space and $T \in B(X)$. Then $T$ is said to be unitary if it satisfy 
\[    \|Tx\| = \|T'Jx\| = \|x\| \,\mbox{ for all } x \in X.\]
\end{definition}
\begin{remark}
    It is well known that a bounded linear operator on a Hilbert space $H$ is normal if and only if $\|Tx\|=\|T^*x\|$ for all $x\in H$ (\cite{limaye1996functional}). If $J$ is the duality map on a Hilbert space $H$, then $T^*=J^{-1}T^{\prime}J$ (\cite{banachSA}), and hence for all $x\in H$, \[\|T^{*}x\|=\|J^{-1}T^{\prime}Jx\|=\|T^{\prime}Jx\|,\] 
    it follows that the Definition \ref{normal-defn} and Definition \ref{unitary-defn} extends the existing concepts of normal and unitary operators to smooth Banach space settings respectively.
\end{remark}
It is obvious that unitary operators are normal. When $T$ is a self-adjoint operator, due to the equality
\[\|T^{\prime}J(x)\|=\|JT(x)\|=\|Tx\|,\] 
it follows that every self-adjoint operator on $X$ is also normal. 

Similar to the classical characterization of unitary operators in the Hilbert space setting, we can characterize the unitary operators on reflexive smooth Banach spaces.

\begin{theorem}\label{unitary_charc}
    Let $X$ be a reflexive smooth Banach space and $T \in B(X)$. Then $T$ is unitary if and only if $T$ is a surjective isometry.
\end{theorem}
\begin{proof}
    If $T$ is unitary, then it is an isometry by the definition. Now, for any $f \in X^*$, there exists an $x\in X$ such that $Jx=f$, $J$ being surjective. So
    \[\|T'f\| = \|T'Jx\| = \|x\| = \|Jx\|=\|f\|.\]
    Hence $T'$ is bounded below, proving that $T$ is surjective (Theorem 13.9, \cite{limaye1996functional}).
\vskip0.15cm
    Conversely, suppose that $T$ is a surjective isometry. Then, since $\|T\|=\|T'\|=1$, for any nonzero $x \in X$,
    \[\|T'Jx\| \leq \|T'\|\|Jx\| = \|Jx\| = \|x\|.\]
    Now, using surjectivity of $T$, choose a $y \in X$ such that $\displaystyle T(y)=\frac{x}{\|x\|}$. Note that $\|y\|=1$, $T$ being an isometry. This gives
    $$
    \|(T'Jx)(y)\| = \|Jx(Ty)\| = \frac{\|Jx(x)\|}{\|x\|}=\|x\|,
    $$
    proving that $ \|(T'Jx)\|\geq \|x\|$. Thus we obtain $\|T^{\prime}Jx\|=\|x\|=\|Tx\|$.
\end{proof}
Now we provide a characterisation of unitary operators, applicable to reflexive smooth strictly convex spaces.

\begin{theorem}
    Let $X$ be a reflexive smooth strictly convex Banach space. Then $T\in B(X)$ is unitary if and only if \[J^{-1}T'JT=TJ^{-1}T'J=I.\]
\end{theorem}
\begin{proof}
        In view of Theorem \ref{unitary_charc}, it is enough to show that $T$ is a surjective isometry if and only if $J^{-1}T'JT=TJ^{-1}T'J=I$.
        
         Assuming $J^{-1}T'JT=TJ^{-1}T'J=I$, for any $x \in X$, we have
    \[\|Tx\|^{2} = J(Tx)(Tx) = T' J(Tx)(x) = J(x)(x) = \|x\|^{2}.\]
This proves that $T$ is an isometry. Now, for any $y \in X$, we have
    \[
    T  J^{-1} T' J(y) = y,
    \]
    which shows that $T$ is surjective.

    Conversely, if we assume that $T$ is an isometry, then for each $x \in X,$
    \[T' J(Tx)(x) = J(Tx)(Tx) = J(x)(x) = \|x\|^2.\]
    Now, for $x\neq 0$, we have, $\displaystyle  \|T'J(Tx)(\frac{x}{\|x\|})\| = \|x\| $, which implies $\|T'J(Tx)\| \geq \|x\|$. On the other hand,
    $$\|T'J(Tx)\| \leq \|T'\|\|JTx\| \leq \|JTx\|=\|Tx\|=\|x\|.$$
   Hence $T'J(Tx)$ satisfies $T' J(Tx)(x) = \|x\|^2$ and $\|T'J(Tx)\| = \|x\|$.  Using the uniqueness of the duality map, we get
   $$
    T'J(Tx) = J(x),
   $$
    and hence, $J^{-1} T'JT = I$ (since $J$ is invertible). In addition, if $T$ is surjective, then for any $y \in X$ there exists an $x \in X$ such that $T(x) = y$. 
    Since $J^{-1} \, T' \, J \, T = I$, we get $T'(J(y)) = J(x),$ 
    Therefore,
    \[
    T J^{-1} \, T' \, J(y) = T J^{-1}(Jx) = y.
    \]
    Hence, we obtain    \(T J^{-1} \, T' \, J = I,\) as required.
\end{proof}

We conclude this section by providing an example of a normal, self-adjoint and unitary operator on $l^4(\mathbb{N})$.
\begin{example}
Consider the operator $T$ on $\ell^4(\mathbb{N})$ defined by
$$T(x)=(x_2,x_1,x_3,x_4,\dots)\, \mbox{ for }\, x=(x_1,x_2,x_3,\dots) \in \ell^4(\mathbb{N}).$$
Consider the transpose $T'$ of this operator defined on $\ell^{4/3}(\mathbb{N})$ by
$$T'(y)=(y_2,y_1,y_3,\dots)\,\mbox{ for }\, y=(y_1,y_2,y_3,\dots)\in \ell^{4/3}(\mathbb{N}).$$ 
The normalized duality map $J:l^4(\mathbb{N}) \to \ell^{4/3}(\mathbb{N})$ is given as 

$$J(x)=\frac{1}{\|x\|^2_4}
\big(|x_1|^2x_1,|x_2|^2x_2,|x_3|^2x_3,\dots\big)$$
Hence, 
\[T'J(x) =\frac{1}{\|x\|^2_4}
\big(|x_2|^2x_2,|x_1|^2x_1,|x_3|^2x_3,\dots\big) \mbox{\quad and \quad}
J(Tx)
=
\frac{1}{\|x\|^2_4}
\big(|x_2|^2x_2,|x_1|^2x_1,\dots\big).
\]
Thus, $T'J(x)=J(Tx),$ so $T$ is self-adjoint, and hence a normal operator. On calculating the norm, we obtain 
$$\|T'J(x)\|_{4/3}^{4/3}=\frac{1}{\|x\|^{8/3}_4}
\left(|x_2|^4+|x_1|^4+\cdots\right)=\|x\|^{4/3}_4.$$
Since $T$ is an isometry on $\ell^4(\mathbb{N})$, we get
$$
\|T'J(x)\|_{4/3}=\|Tx\|_4 =\|x\|_4 \quad \text{for all } x\in \ell^4(\mathbb{N}),
$$ showing that $T$ is a unitary operator.

\end{example}

\section{Spectral properties of self-adjoint and normal operators}
In this section, we establish some results on the spectrum of self-adjoint operators and normal operators on smooth Banach spaces. In the Hilbert space setting, it is well known that for a self-adjoint operator, the eignevalues are real, and the eigenvectors corresponding to distinct eigenvalues are orthogonal. In the setting of smooth Banach spaces, it was shown (Theorem 6.3(i)) in \cite{banachSA} that the eigenvalues of self-adjoint operators are real. Here, we obtain an analogous result equivalent to orthogonality of eigenvectors in the setting of smooth Banach spaces.
\begin{proposition}\label{eigenvector_perp}
    Let $x_1$ and $x_2$ be eigenvectors corresponding to two distinct eigenvalues of a self-adjoint operator $T$ on smooth Banach space $X$. Then $x_2 \in {}^\perp J(x_1)$ and $x_1 \in {}^\perp J(x_2)$ .
\end{proposition}
\begin{proof}
    Let $Tx_1=k_1x_1$ and $Tx_2=k_2x_2$, where $k_1\neq k_2$. Since $k_1$ and $k_2$ are real numbers, we have
   \[k_1J(x_1)(x_2) = J(k_1x_1)(x_2)=J(Tx_1)(x_2)= J(x_1)(Tx_2)=k_2J(x_1)(x_2).\]
 Hence, if $k_1 \neq k_2$, then $x_1 \in {}^\perp J(x_2)$. Similarly, we get $x_2 \in {}^\perp J(x_1)$.
\end{proof}
 It is well known that the eigenspectrum of a self-adjoint operator on a separable Hilbert space is countable. In Proposition \ref{lemma-martin}, we obtain a similar result in the setting of smooth Banach spaces. In order to establish this, we recall a known result from \cite{martin2024}.
\begin{proposition}[Fact 1.1 \cite{martin2024}]\label{birkoff}
    For any two elements $x,y$ in a Banach space $X$, $\|x+\lambda y\|\geq \|x\|$ for every $\lambda \in \C$ if and only if there exists a $\phi \in S_{X^*}$ such that $\phi(x)=\|x\|$ and $\phi(y)=0$.
\end{proposition}
Now we prove the countability of the eigenspectrum of self-adjoint operators on smooth Banach spaces.
\begin{proposition}\label{lemma-martin}
    Let $X$ be a separable smooth Banach space, and $T \in B(X)$ be self-adjoint. Then $\sigma_{eig}(T)$ is countable.
\end{proposition}
\begin{proof}
    Let $\Omega$ be the set constructed by taking one unit eigenvector corresponding to each distinct eigenvalue of $T$. Then for each $x\neq y$ from $\Omega$, by Proposition \ref{eigenvector_perp}, we have, $x\in {}^{\perp}J(y)$. This means $J(x)(y)=0$ and $\|J(y)\|=1=\|y\|$.  Now applying Proposition \ref{birkoff} with $\lambda=-1$, we get $\|x-y\|\geq 1$. Thus the elements of $\Omega$ are $1-$distance apart from each other.

    Now suppose that $\Omega$ is uncountable. Then, consider the collection 
    $$A=\{B(x, \frac{1}{2}):\, x\in \Omega\},$$ 
    where $B(a,r)$ represent an open ball centred at $a$ of radius $r$. Then $A$ is an uncountable family of pairwise disjoint open balls in $X$. Consequently, any dense subset of the space must intersect each element of $A$, forcing every dense subset to be uncountable, which contradicts the separability of the space. Therefore, $\Omega$ must be countable.
\end{proof}
We can establish the equality of the approximate spectrum with the spectrum for self-adjoint operators, in certain cases.
\begin{theorem}
    Let $X$ be a reflexive smooth strictly convex Banach space, and let $T \in B(X)$ be a self-adjoint operator. Then $\sigma(T) = \sigma_{app}(T)$
\end{theorem}
\begin{proof}
    We have as mentioned earlier $\sigma(T) = \sigma_{app}(T) \cup \sigma_{eig}(T')$. We prove the result by showing $\sigma_{eig}(T')\subseteq \sigma_{eig}(T)$.  Consider $\lambda \in \sigma_{eig}(T')$. Then there exists a nonzero $f \in X^*$ such that $\|(T'-\lambda I)f\|=0$. Since $J$ is onto, we get a nonzero $x\in X$ such that $\|(T'-\lambda I)Jx\|=0$, giving $T'Jx = \lambda Jx$. Thus,
    $$
    J(\overline{\lambda}x)=\lambda J(x) = T'Jx = J(Tx).
    $$
    The injectivity of $J$ implies $Tx = \overline{\lambda }x$. Since the eigenvalues of $T$ are real, this gives $\lambda \in \sigma_{eig}(T)$. Thus $\sigma(T) = \sigma_{app}(T)$.
\end{proof}

It is a classical result in the Hilbert space context that the approximate spectrum coincides with the spectrum for a normal operator also. The following result provides a partial extension of this property to reflexive smooth Banach spaces.

\begin{theorem}\label{spectrum=approx}
Let $X$ be a reflexive smooth Banach space and $T\in {B}(X)$. 
If $\lambda \in \sigma(T)$ is such that $T-\lambda I$ is normal, then $\lambda \in \sigma_{app}(T)$.
\end{theorem}
\begin{proof}
    Consider $\lambda \in \sigma(T)= \sigma_{app}(T)\cup \sigma_{eig}(T')$. If $\lambda \in \sigma_{eig}(T')$, then we obtain a nonzero $f \in X^*$ such that $\|(T'-\lambda I)f\|=0$. Now, as in the earlier case, since $J$ is onto, there is an $x\neq 0$ in $X$ such that
    \[\|(T'-\lambda I)f\|=\|(T-\lambda I)'Jx\|=\|(T-\lambda I)x\|,\]
    since $T-\lambda I$ is normal. This proves that $\lambda \in \sigma_{eig}(T)\subseteq \sigma_{app}(T)$.
\end{proof}

Note that in the Hilbert space setting, if $T$ is normal, then $T-\lambda I$ is normal for any scalar $\lambda$. As a direct consequence of the above result, we obtain the equality of the minimum modulus and Crawford number for a non-invertible normal operator, generalizing Proposition 3.15 of \cite{crawford2024}.

\begin{corollary}\label{non_inv_bound}
    Let $T$ be a normal operator on a reflexive smooth Banach space $X$. Then $T$ is invertible if and only if $T$ is bounded below. Hence, if $T$ is a non-invertible normal operator, then  $c(T)=\mu(T)=0.$
\end{corollary}
\begin{proof}
    The first part follows from Theorem \ref{spectrum=approx} by taking $\lambda=0$, and the second from $0\leq c(T)\leq \mu(T)=0$, when $T$ is not invertible. 
\end{proof}

\section*{Conclusion and open problems}

In this article, we have considered three classes of operators on smooth Banach spaces - namely self-adjoint, normal, and unitary operators - and investigated several of their properties, including attainment properties and spectral behaviour. In particular, we obtained a number of results that are analogous to their counterparts in the Hilbert space setting. At the same time, the extent to which other well-known properties from the Hilbert space theory remain valid in the setting of smooth Banach spaces is not yet fully clear and needs further investigation. This leads to several natural questions. 
\begin{enumerate}
    \item Under what conditions does a self-adjoint operator become Hermitian and conversely?
    \item While strongly normal operators are known to be self-adjoint and positive, is the converse true, and can one obtain a characterisation of strongly normal operators? 
    \item If $T$ is a self-adjoint operator on a smooth Banach space, are the operators $
\|T\|^2 I - T^2 \quad \text{and} \quad T^2 - \mu(T)^2 I$ strongly normal, or can counterexamples be found? 
\item What can be said about the structure of the numerical range of special operators such as normal, unitary, and self-adjoint operators, and under what conditions is it convex? 
\item Is there any relationship between the operator identity $T J^{-1} T' J = J^{-1} T' J T$ and the norm equality $\|Tx\| = \|T'Jx\| $ for all $ x \in X$, which are equivalent in the Hilbert space setting? 
\item While the spectral radius, numerical radius, and operator norm coincide for normal operators on Hilbert spaces, does an analogous result hold in the present framework? 
\item In separable Hilbert spaces, the eigenspectrum of a normal operator is countable and that of a unitary operator lies on the unit circle; what can be said about these properties in the smooth Banach space setting?
\end{enumerate}

\section*{Acknowledgments}
 Mr. Mohammed Shameem would like to thank the University Grants Commission of India for funding his research work (Fellowship No. 231610086159). The authors thank Ms. Sreelakshmi M. P. for early discussions on duality mappings and smooth Banach spaces.

\section*{Statements and Declarations}
\subsection*{Conflict of Interest}
 The corresponding author declares that there is no conflict of interest on behalf of all authors.

\end{document}